\documentclass[final,leqno,onefignum,onetabnum]{siamltex1213}
\newtheorem{remark}[theorem]{Remark}
\usepackage{color,amssymb,amsmath,slashbox}
\usepackage[usenames,dvipsnames,svgnames,table]{xcolor}

\title{Condition number estimates for matrices arising in NURBS based
isogeometric discretizations of elliptic partial differential equations\thanks{This research was supported in part by
the Austrian Sciences Fund (Project P21516-N18),
National Science Foundation grant EPS-1135483,
Award No. KUS-C1-016-04, made by King Abdullah University of Science \&
Technology (KAUST),
and
the KAUST Numerical Porous Media Center.}}

\author{Krishan P. S. Gahalaut\thanks{Division of Computer, Electrical and Mathematical Sciences and
Engineering,
King Abdullah University of Science and Technology,
Thuwal 23955-6900, Kingdom of Saudi Arabia.
(\email{krishan.gahalaut@kaust.edu.sa}). Questions, comments, or corrections
to this document may be directed to that email address.}\and{Satyendra K. Tomar\thanks{Dornacher Strasse 6/21, 4040 Linz, Austria.
(\email{tomar.sk.prof@gmail.com}).}}\and{Craig C. Douglas\thanks{School of Energy Resources and Mathematics Department,
University of Wyoming, Laramie, WY 82071-3036, USA.
(\email{craig.c.douglas@gmail.com}).}}}

\begin{document}
\maketitle
\slugger{sinum}{xxxx}{xx}{x}{x--x}


\begin{abstract}
We derive bounds for the minimum and maximum eigenvalues and the spectral
condition number of matrices for isogeometric discretizations of elliptic
partial differential equations in an open, bounded, simply connected Lipschitz
domain $\Omega\subset \mathbb{R}^d$, $d\in\{2,3\}$.
We consider refinements based on mesh size $h$ and polynomial degree $p$ with 
maximum regularity of spline basis functions.  For the $h$-refinement, the
condition number of the stiffness matrix is bounded above by a constant times
$ h^{-2}$ and the condition number of the mass matrix is uniformly bounded.
For the $p$-refinement, the condition number grows exponentially and is
bounded above by $p^{2d+2}4^{pd}$ and $p^{2d}4^{pd}$ for the stiffness
and mass matrices, respectively.
Rigorous theoretical proofs of these estimates and supporting
numerical results are provided.
\end{abstract}


\begin{AMS}\end{AMS}


\pagestyle{myheadings}
\thispagestyle{plain}
\markboth{Gahalaut, Tomar and Douglas}{Condition number estimates for matrices arising in NURBS based discretizations}


\section{Introduction}
\label{Sec:Introduction}

Isogeometric analysis is a term introduced by Hughes et al. in 2005
\cite{HughesCB-05}.
Most of the research activity in isogeometric analysis has focused on using
Non-Uniform Rational B-Spline (NURBS) as basis functions, e.g.,
\cite{AVBLovadinaRS-07, BVCHughesS-06, CHughesB-09, HughesCB-05}.
Isogeometric analysis is not restricted to NURBS basis functions.
Other types of basis functions are used by researchers, e.g., T-Splines,
hierarchical B-Splines, and subdivision schemes.
Use of splines as finite element basis functions dates back to the 1970's,
however \cite{Prenter-75, Schultz-73, Weiser-80}.

In isogeometric analysis the computational geometry (e.g., a circle) is
represented exactly from the information and the basis functions given by
Computer Aided Design (CAD).
It holds an advantage over classical finite element methods (FEM), where the
basis functions are defined using piecewise polynomials and the computational
geometry (i.e., a mesh) is defined on polygonal elements.
It has been argued in \cite{CHughesB-09} that NURBS based isogeometric method
leads to qualitatively more accurate results than a standard piecewise
polynomial based finite element method.
Typically, the solution computed by an isogeometric method has a higher
continuity than the one computed in a classical finite element method.
It is a difficult and cumbersome task to achieve even $C^{1}$ inter-element
continuity in the piecewise polynomial based finite element method, whereas
isogeometric method offers up to $C^{p-m}$ continuity, where $p$ denotes the
degree of the basis functions and $m$ denotes the knot-multiplicity.
Finally, isogeometric analysis provides a powerful tool to compute highly
continuous numerical solution of PDEs arising in engineering sciences.

Since the introduction of isogeometric analysis, most of its progress has been
focused on the applications and discretization properties.
Nevertheless, when dealing with large problems, the cost of solving the linear
system of equations arising from the isogeometric discretization becomes an
important issue.
Clearly, the discretization matrix $A$ gets denser by increasing the
polynomial degree $p$.
Therefore, the cost of a direct solver, particularly for large problems,
becomes prohibitively expensive.
The most practical way to solve them is to resort to an iterative method.
Since the convergence rate of such methods is strongly affected by the
condition number of the system matrix $A$, it is important to assess this
quantity as a function of the mesh size $h$ for the $h$-refinement, or as a
function of the degree $p$ for the $p$-refinement.
Note that in the $p$-refinement, improved approximate solutions are sought by
increasing $p$ while the mesh of the domain, and thus the maximum
quadrilateral diameter $h$, is held fixed, whereas in the $h$-refinement,
improved approximations are obtained by refining the mesh, and thus reducing
$h$, while $p$ is held fixed.
In this paper we consider both the cases: $h$- and $p$-refinements.
Similar efforts are made in \cite{Garoni-13} on the spectrum of stiffness
matrices and in \cite{Pilgerstorfer-14} on bounding the influence of the
domain parameterization and knot spacing.
However, these papers primarily derive bounds with respect to the mesh size
$h$.
To the best of our knowledge there is no study that discusses the bounds on
condition number estimates of isogeometric matrices with respect to
$p$-refinement.
Our main results provide upper bounds for the condition number of the
stiffness matrix and the mass matrix for both the $h$- and $p$-refinements.

For $h$-refinement applied to second order elliptic problems on a regular
mesh, the condition number of the finite element stiffness matrix scales as
$h^{-2}$ and the condition number of the mass matrix is bounded uniformly,
independent of $h$ \cite{Babuska-89}.
This is true for a great variety of elements and independent of the dimension
of the problem domain.
Our results here are in agreement \cite{Axelsson-01} and are useful in
theoretical analysis that relates to $h$-refinement.
For example, in convergence analysis of multigrid methods, these results are
one of the key elements in deriving convergence factors, for finite element
analysis \cite{Braess-07, Hackbusch-94, Saad-96} and for isogeometric analysis
\cite{Gahalaut-12}.

The order of the approximation error of the numerical
solution depends on the choice of the finite dimensional subspace, not on
the choice of its basis \cite{Ciarlet-78}.
Therefore, when working with a finite element method or an isogeometric method
for elliptic problems, we only consider function spaces rather than the
choice of particular basis functions.
Nevertheless, the choice of the basis functions affects the condition number
of the stiffness and the mass matrices, which influences the performance of
iterative solvers.
There is no general theory to characterize the extremal eigenvalues or the
condition number based on a set of general polynomial basis functions
\cite{Babuska-81, Maitre-96, Melenk-02, Melenk-96}.
Unlike the $h$-refinement case, the condition number heavily depends on the
choice of basis functions for the $p$-refinement.

For different choices of basis functions the condition number may grow
algebraically or exponentially.
Olsen and Douglas \cite{Olsen-95} estimated the condition number bounds of
finite element matrices for tensor product elements with two choices of basis
functions.
For Lagrange elements, it is proved that the condition number grows
exponentially in $p$.
For hierarchical basis functions based on Chebychev polynomials, the condition
number grows rapidly but only algebraically in $p$.
Similar results on the condition number bounds can be found in \cite{Ern-91,
Hu-98, Maitre-95}.

Due to the larger support of NURBS basis functions, the band of the stiffness
matrix corresponding to the NURBS-based isogeometric method is less sparse
than the one arising from piecewise polynomial finite element procedures.
Therefore, a larger condition number is expected.
Our results for the $p$-refinement case show that the condition number of
system matrices in an isogeometric method grows exponentially.

Throughout this paper we deal with the maximum regularity $C^{p-1}$ of a
B-spline unless otherwise specified.
The generic constant $C$, which will be used often takes different values at
different occasions, and is independent of $h$ and $p$ in the analysis with
respect to $h$-refinement and $p$-refinement, respectively.
Moreover, in our numerical studies the coarsest and finest meshes use $h=1$
and $h=1/128$, respectively.

The remainder of the paper   is organized as follows. 
In Section \ref{Sec:ModelProb}, we describe the model problem and its
discretization.
In Section \ref{Sec:SplineCondNum}, we define B-Splines and NURBS and a basic
notation.
We recall bounds for the condition number of a B-Spline basis function.
In Section \ref{Sec:StiffnessCondNum}, we derive bounds for the eigenvalues and
the condition number of the stiffness and mass matrices arising in
isogeometric discretizations for the $h$- and $p$-refinement cases.
In Section \ref{NumericalResults}, we provide numerical experiments that
support the theoretical estimates.
In Section~\ref{Sec:Conclusions}, we draw some conclusions and discuss future
work.


\section{Model problem and its discretization}
\label{Sec:ModelProb}

Let $\Omega\subset\mathbb{R}^{d}$, $d\in\{2,3\}$, be an open, bounded, and
simply connected Lipschitz domain with Dirichlet boundary $\partial \Omega$.
We consider the Poisson equation,
\begin{subequations}
\label{eq:Laplace}
\begin{alignat}{2}
\Delta u &= -f \quad \mathrm{in~} \quad  \Omega, \\
 u &= 0  \quad \mathrm{on~}\quad  \partial \Omega,
\end{alignat}
\end{subequations}
where $f:\Omega \rightarrow \mathbb{R}$ is given.
The aim is to find $u:(\Omega \cup \partial \Omega) \rightarrow \mathbb{R}$
that satisfies \eqref{eq:Laplace}.
We consider Galerkin's formulation of the problem, which is commonly used in
isogeometric analysis.
Since we are interested in the study of the condition number, therefore we
shall not go into the details of the solution properties, and restrict
ourselves to the study of the condition number of the resulting system
matrices.

Isogeometric analysis has the same theoretical foundation as finite element
analysis, namely the variational form of a partial differential equation.
We define the function space $\mathcal{S}$ as all the functions that have
square integrable derivatives and also satisfy $u|_{\partial \Omega}=0$,
\begin{equation}
\mathcal{S}= \{u:u\in H^1(\Omega),u|_{\partial \Omega}=0\},
\end{equation}
where $H^1(\Omega)=\{u:D^{\alpha}u\in L^2(\Omega),|\alpha|\leq1\}$ is a
Sobolev space, $\alpha \in \mathbb{N}^d$ is a multi-index,
$D^{\alpha}=D_1^{\alpha_1}D_2^{\alpha_2} \ldots  D_d^{\alpha_d}$, and
$D_i^j=\displaystyle \frac{\partial^j}{\partial x_i^j}$.

We write the variational formulation of the model problem by multiplying
it by an arbitrary function $v\in \mathcal{S}$ and integrating by parts.
For a given $f$: find $u\in \mathcal{S}$ such that for all $ v\in
\mathcal{S}$,
\begin{equation*}
\int_{\Omega} \nabla u\cdot \nabla v\hspace{.15cm}d\Omega = \int_{\Omega}  f v \hspace{.15cm}d\Omega.
\end{equation*}
We rewrite the formulation as: find $u\in \mathcal{S}$ such that for all $
v\in \mathcal{S}$,
\begin{equation}
a(u,v)=L(v),\quad \forall v\in\mathcal{S},
\end{equation}
where
\begin{equation*}
a(u,v)=\int_{\Omega} \nabla u\cdot \nabla v\hspace{.15cm}d\Omega, \quad
\text{ and } \quad
L(v)=\int_{\Omega} f v \hspace{.15cm}d\Omega.
\end{equation*}
Note that $a(\cdot,\cdot)$ is a bilinear form that is continuous and coercive
on $\mathcal{S}$.
$L(\cdot)$ is a linear form associated with the original equation.

Let $\mathcal{S}^h\subset\mathcal{S}$ be a finite dimensional approximation
of $\mathcal{S}$.
The Galerkin form of the problem is:  Find $u^h \in \mathcal{S}^h$ such that
for all $\text{ } v^h \in \mathcal{S}^h$,
\begin{equation}
\label{eq:weak_form_h}
a(u^h,v^h)=L(v^h),
\end{equation}
which is a well-posed problem with a unique solution \cite{Ciarlet-78}.

By approximating $u_h$ and $v_h$ using spline (see Section
\ref{Sec:SplineCondNum}) basis functions $N_{i}$, $i=1, 2, \ldots, 
n_h = \mathcal{O}(h^{-2})$, the variational formulation
\eqref{eq:weak_form_h} is
transformed into a set of linear algebraic equations,
\begin{equation}
\label{eq:LinSys}
A \bf u=f.
\end{equation}
$A$ denotes the stiffness matrix obtained from the bilinear form
$a(\cdot,\cdot)$,
\begin{equation*}
A = (a_{i,j}) = (a(N_i,N_j)), \quad i,j=1,2,3,\ldots.,n_h.
\end{equation*}
$\bf u$ denotes the vector of unknown degrees of freedom and $\bf f$ denotes
the right hand side vector from the known data of the problem.
$A$ is a real, symmetric positive definite matrix.


\section{Splines and their condition number bounds}
\label{Sec:SplineCondNum}

Non-uniform rational B-Splines (NURBS) are commonly used in isogeometric
analysis and are built from B-Splines.
In Section \ref{Subsec:B-Splines+NURBS}, we give a brief description of
B-Splines and NURBS and their properties.
In Section \ref{Subsec:DerivativesOfBSplines}, we define the derivatives of
B-Splines.
In Section \ref{Subsec:CondNumOfBSplines}, we prove bounds on the condition
number of B-Spline basis functions.
 

\subsection{B-Splines and NURBS}
\label{Subsec:B-Splines+NURBS}

In this section, we define B-Spline and NURBS functions.
We also define surfaces and describe higher order objects based on both types
of functions.

The Cox-de Boor reursion formula \cite{Boor-78} is given by
\begin{definition}
\label{Def:BSplines}
Let $\Xi_{1}=\{\xi_i :  i = 1, \ldots, n + p+1\}$ be a non-decreasing sequence of
real numbers called the $knot$ $vector$, where $\xi_i$ is the $i^{th}$ knot,
$p$ is the polynomial degree, and $n$ is the number of basis function.
With a knot vector in hand, the B-Spline basis functions denoted by $
N^p_{i}(\xi)$ are (recursively) defined starting with a piecewise constant
$(p=0)$:
\begin{subequations}
\label{eq:DefBSplines}
\begin{alignat}{3}
N^0_{i}(\xi) &= \begin{cases} 1 & \text{if $\xi \in [\xi_i,\xi_{i+1})$,} \\ 0 &\text{otherwise,} \end{cases}\\
\quad N^p_{i}(\xi) & =\frac{\xi-\xi_i}{\xi_{i+p}-\xi_i} N^{p-1}_{i}(\xi) +\frac{\xi_{i+p+1}-\xi}{\xi_{i+p+1}-\xi_{i+1}} N^{p-1}_{i+1}(\xi),
\end{alignat}
\end{subequations}
where $ 0 \leq i \leq n, p \geq 1$ and $\displaystyle \frac{0}{0}$ is considered as zero.
\end{definition}

For a B-Spline basis function of degree $p$, an interior knot can be repeated
at most $p$ times, and the boundary knots can be repeated at most $p+1$ times.
A knot vector for which the two boundary knots are repeated $p+1$ times is
said to be open.
In this case, the basis functions are interpolatory at the first and the last
knot.
Important properties of the B-Spline basis functions include nonnegativity,
partition of unity, local support and $C^{p-k}$-continuity.
 
Higher dimensional B-Spline objects are defined using tensor products.
\begin{definition}
\label{Def:BSplineCurve}
A B-Spline curve $C(\xi)$ is defined by
\begin{equation}
C(\xi)=\sum^{n}_{i=1} P_{i} N^p_{i}(\xi),
\end{equation}
where $\{P_i: i = 1, \ldots, n\}$ are the control points and $N^p_{i}$ are
B-Spline basis functions defined in (\ref{eq:DefBSplines}).
\end{definition}
\begin{definition}
\label{Def:BSplineSurface}
A B-Spline surface
$S(\xi,\eta)$ is defined by
\begin{equation}
\label{eq:BSurface}
S(\xi,\eta) = \sum_{i=1}^{n_{1}} \sum_{j=1}^{n_{2}} N_{i,j}^{p_{1},p_{2}} (\xi,\eta) P_{i,j},
\end{equation}
where $P_{i,j}$, $i = 1,2, \ldots, n_{1}$, $j = 1,2, \ldots, n_{2}$, denote
the control points, $N_{i,j}^{p_{1},p_{2}}$ is the tensor product of B-Spline
basis functions $N_{i}^{p_{1}}$ and $N_{j}^{p_{2}}$, and $\Xi_{1} = \{\xi_{1},
\xi_{2}, \ldots, \xi_{n_{1}+p_{1}+1}\}$ and $\Xi_{2} = \{\eta_{1}, \eta_{2},
\ldots, \eta_{n_{2}+p_{2}+1}\}$ are the corresponding knot vectors.
\end{definition}

Similarly three dimensional B-Spline solids can be defined using two tensor
products.

Polynomials cannot exactly describe frequently encountered shapes in
engineering, particularly the conic family (e.g., a circle).
While B-Splines are flexible and have many nice properties for curve design,
they are also incapable of representing such curves exactly.
Such limitations are overcome by NURBS functions that can exactly represent a
wide array of objects.

Rational representation of conics originates from projective geometry.
The ``coordinates'' in the additional dimension are called weights, which we
shall denote by $w$.
Furthermore, let $\{P^{w}_{i}\}$ be a set of control points for a projective
B-Spline curve in $\mathbb{R}^{3}$.
For the desired NURBS curve in $\mathbb{R}^{2}$, the weights and the control
points are derived by the relations
\begin{equation}
\label{eq:NWts}
w_{i} = (P^{w}_{i})_{3}, \qquad (P_{i})_{d} = (P^{w}_{i})/w_{i}, \quad d = 1,2,
\end{equation}
where $w_{i}$ is called the $i^{\mathrm{th}}$ weight and $(P_{i})_{d}$ is the
$d^{\mathrm{th}}$-dimension component of the vector $P_{i}$.
The weight function $w(\xi)$ is defined as
\begin{equation}
\label{eq:NWFunc}
w(\xi) = \sum_{i = 1}^{n}N_i^p(\xi) w_{i}.
\end{equation}

We now formally define NURBS objects.
\begin{definition}
\label{Def:NURBSCurve}
The NURBS basis functions and curve are defined by
\begin{equation}
\label{eq:NCurve}
R_i^p(\xi)  = \frac{N_i^p(\xi) w_{i}}{w(\xi)}\quad 
{\rm and}\quad
C(\xi) = \sum_{i=1}^{n} R_i^p(\xi) P_{i}.
\end{equation}
\end{definition}
\begin{definition}
\label{Def:NURBSSurface}
The NURBS surfaces are defined by
\begin{equation}
\label{eq:NSurface}
S(\xi,\eta) = \sum_{i=1}^{n_{1}} \sum_{j=1}^{n_{2}} R_{i,j}^{p_{1}, p_{2}}
(\xi,\eta) P_{i,j},
\end{equation}
where $R_{i,j}^{p_{1},p_{2}}$ is the tensor product of NURBS basis functions
$R_{i}^{p_{1}}$ and $R_{j}^{p_{2}}$.
\end{definition}
\noindent
NURBS functions also satisfy the properties of B-Spline functions
\cite{PieglTiller-97, Rozers-01, Schumaker-07}.


\subsection{Derivatives of B-Splines}
\label{Subsec:DerivativesOfBSplines}

Derivatives of B-Splines \cite{Floater-92} and their conditioning are very
important for the estimation of the condition number of the stiffness matrix.
The recursive definition of B-Spline functions allow us to seek the
relationship between the derivative of a given B-Spline basis function and
lower degree basis function.
\begin{definition}
\label{def:DerBSplines}
The derivative of the $i^{th}$ B-Spline basis function defined in
(\ref{eq:DefBSplines}) is given by
\begin{equation}
\label{eq:BsplineDer}
\frac{d}{d\xi} N^p_{i}(\xi) =\frac{p}{\xi_{i+p}-\xi_i} N^{p-1}_{i}(\xi) -
\frac{p}{\xi_{i+p+1}-\xi_{i+1}} N^{p-1}_{i+1}(\xi).
\end{equation}
\end{definition}
\noindent
By repeated differentiation of (\ref{eq:BsplineDer}) we get the general
formula for any order derivative.
Since we are only interested in the first derivative, we ignore further
details \cite{Boor-78}.

The derivatives of rational functions will clearly depend on the derivatives
of their non-rational counterpart.
Definition \ref{def:DerBSplines} can be generalized for NURBS.
\begin{definition}
\label{def:DerNurbs}
The derivative of the $i^{th}$ NURBS basis function is given by
\begin{equation}
\label{Eq:NURBSDer}
\displaystyle\frac{d}{d\xi} R^p_{i}(\xi) = w_i \displaystyle \frac {w(\xi)
\displaystyle\frac{d}{d\xi}N^{p}_{i}(\xi) - \displaystyle \frac{d}{d\xi}w(\xi)
N^{p}_{i}(\xi)}{(w(\xi)^2)}.
\end{equation}
where $w_i$ and $w(\xi)$ are defined in (\ref{eq:NWts}) and (\ref{eq:NWFunc}),
respectively.
\end{definition}


\subsection{Condition number of B-Splines}
\label{Subsec:CondNumOfBSplines}

In this section, we recall bounds for the condition number of  B-Splines.

We need to know the bounds on B-Spline basis functions in some $L_{s}$-norm,
where $s \in [1,\infty]$.
We estimate the size of the coefficients of a polynomial of degree $p$
in two dimensions when it is represented using the tensor product structure of B-Spline basis functions.
The condition number of a basis can be defined as follows.
\begin{definition}
A basis $\{N_i\}$ of a normed linear space is said to be stable with respect
to a vector norm if there are constants $K_1$ and $K_2$ such that for all
coefficients $\{v_i\}$ the following relation holds:
\begin{equation}
\label{eq:BasisCondP}
K_1^{-1}\|\{v_i\}\| \leq \Big\|\sum_{i} v_i N_i \Big\|\leq K_2\|\{v_i\}\|.
\end{equation}
The number $\kappa=K_1K_2$, with $K_1$ and $K_2$ as small as possible, is
called the condition number of $\{N_i\}$ with respect to $\|\cdot \|$.
Note that we use the symbols $\|\cdot \|$ and $\|\{\cdot \}\|$ for the norms
in the vector space and the vector norm, respectively.
\end{definition} 

Such condition numbers give an upper bound for magnification of the error in
the coefficients to the function values.
Indeed, if $\displaystyle f=\sum_{i} f_i N_i \neq 0$ and $\displaystyle
g=\sum_{i} g_i N_i $, then it follows immediately from \eqref{eq:BasisCondP}
that
\begin{equation*}
\displaystyle \frac{\|f-g\|}{\|f\|}\leq \kappa
\frac{\|\{f_i-g_i\}\|}{\|\{f_i\}\|}.
\end{equation*}
More details on the approximation properties and the stability of B-Splines
can be found in \cite{Hoellig-82, Hoellig-00, Hoellig-02, Lyche-97, Lyche-00,
MoessnerR-08, Pena-97}.
We use these estimates of $\kappa$ to estimate the bounds of the condition
number of the stiffness matrix and the mass matrix.

It is of central importance for working with B-Spline basis functions that its
condition number is bounded independently of the underlying knot sequence.\
That is, the condition number of B-Splines does not depend on the multiplicity
of the knots of knot vector \cite{Boor-72, Boor-76, Boor-76-2, Boor-74}.
In \cite{Boor-76} is a direct estimate that the worst condition number of a
B-Spline of degree $p$ with respect to any $L_s$-norm is bounded above by
$p9^p$.
It is also conjectured that the real value of $\kappa$ grows like $2^p$, which
is superior to the direct estimate:
\begin{subequations}
\begin{alignat}{3}
\kappa & < p9^p & \quad \text{(direct estimate)},\\
\label{eq:CondNumConjecture}
\kappa & \sim2^p & \quad \text{(de Boor's conjecture)}.
\end{alignat}
\end{subequations}
In \cite{Boor-90}, the exact condition number of a B-Spline basis is shown to
be difficult to determine.

Scherer and Shadrin \cite{Scherer-96} proved that the upper bound of the
condition number $kappa$ of a B-Spline of degree $p$ with respect to
$L_s$-norm is bounded by
\begin{equation}
\kappa  < \displaystyle p^{\frac{1}{2}} 4^p,
\end{equation}
which is closer to the conjecture in (\ref{eq:CondNumConjecture}).
Scherer and Shadrin \cite{Scherer-99} proved the following result.
\begin{lemma}
\label{Lemma:SS99}
For all $p$ and all $s \in [1,\infty]$,
\begin{equation}
\label{eq:kappas}
\kappa < p2^p.
\end{equation}
\end{lemma}
\noindent
Lemma~\ref{Lemma:SS99} confirms the conjecture
(\ref{eq:CondNumConjecture}) up to a polynomial factor.
Possible approaches to eliminate the polynomial factor are also discussed in
\cite{Scherer-99}.
Lemma~\ref{Lemma:SS99} can be easily generalized to $d$-dimensions.
\begin{lemma}
\label{Lemma:SS99-d-dims}
Using a tensor product B-Spline basis of degree $p$ in $d$-dimensions and
\eqref{eq:kappas}, the following is immediate:
\begin{equation}
\label{eq:BspKapD}
\kappa < (p2^p)^d.
\end{equation}
\end{lemma}


\section{Estimates of condition number}
\label{Sec:StiffnessCondNum}

In this section, we give estimates for the condition number of the stiffness
matrix (in Section~\ref{Subsec:CondNum4StiffnessMatrix}) and the mass matrix
(in Section~\ref{Subsec:CondNum4MassMatrix}) obtained from isogeometric
discretization.
In each case, we have bounds on the condition number with respect to both $h$-
and $p$-refinements.
For $h$-refinement, upper bounds for the maximum eigenvalues, a lower bound for
the minimum eigenvalue, and an upper bound for the condition number are given.
For $p$-refinement, we prove upper and lower bounds for the maximum eigenvalue,
lower bounds for the minimum eigenvalue, and upper bounds for the condition
number.


\subsection{Stiffness matrix}
\label{Subsec:CondNum4StiffnessMatrix}

In this section, we give estimates for the condition number of the stiffness
matrix with estimates for $h$-refinement in
Section~\ref{Subsec:CondNum4StiffnessMatrix-h} and for $p$-refinement in
Section~\ref{Subsec:CondNum4StiffnessMatrix-p}.


\subsubsection{$h$-refinement}
\label{Subsec:CondNum4StiffnessMatrix-h}


Without loss of generality, we begin with a two-dimensional open parametric
domain $\Omega= (0,1)^{2}$ that we refer to as a \textit{patch}.
Given two open knot vectors $ \Xi_{1} = \{ 0 = \xi_1, \xi_2, \xi_3, \ldots,
\xi_{m_{1}} = 1 \}$ and $ \Xi_{2} = \{ 0 = \eta_1, \eta_2, \eta_3, \ldots,
\eta_{m_{2}} = 1 \}$, we partition the patch $\Omega$ into a mesh
\begin{equation*}
\mathcal{Q}_h = \{ Q = (\xi_{i}, \xi_{i+1}) \otimes (\eta_{j}, \eta_{j+1}), i
= p_{1}+1, 2, \ldots, m_{1}-p_{1}-1, j = p_{2}+1, 2, \ldots, m_{2}-p_{2}-1\},
\end{equation*}
where $Q$ is a two-dimensional open knot-span whose diameter is denoted by
$h_{Q}$.
We consider a family of quasi-uniform meshes $\{\mathcal{Q}_h\}_{h}$ on
$\Omega$, where $h = \max \{h_{Q} | Q \in \mathcal{Q}_h \}$ denotes the family
index \cite{BVCHughesS-06}.
Let $\mathcal{S}_h$ denote the B-spline space associated with the mesh
$\mathcal{Q}_h$.
Given two adjacent elements $Q_{1}$ and $Q_{2}$, we denote by $m_{Q_{1}Q_{2}}$ 
the number of continuous derivatives across their common face $\partial Q_{1}
\cap \partial Q_{2}$.
In the analysis, we will use the usual Sobolev space of order $m \in
\mathbb{N}$,
\begin{align}
\label{eq:BSspace}
\mathcal{H}^m(\Omega) = & \Big \{ v \in L^{2}(\Omega) \mathrm{~such~that~}
v|_{Q} \in H^{m}(Q), \forall Q \in \mathcal{Q}_h, \mathrm{~and} \\
& \nabla^{i}(v|_{Q_{1}}) = \nabla^{i}(v|_{Q_{2}}) \mathrm{~on~} \partial Q_{1}
\cap \partial Q_{2}, \nonumber\\
& \forall i \in \mathbb{N} \mathrm{~with~} 0 \le i \le \min \{m_{Q_{1}Q_{2}},
m-1 \}, \forall Q_{1}, Q_{2} \mathrm{~with~} \partial Q_{1} \cap \partial
Q_{2} \neq \emptyset \Big \}, \nonumber
\end{align}
where $\nabla^{i}$ has the usual meaning of $i^{\mathrm{th}}$-order partial
derivative.
The space $\mathcal{H}^m$ is equipped with the following semi-norms and norm
\begin{equation*}
\vert v \vert_{\mathcal{H}^i(\Omega)}^{2} = \sum_{Q \in \mathcal{Q}_h} \vert
v \vert_{H^{i} (Q)}^{2},~~
0 \le i \le m,\quad {\rm and}\quad
\Vert v \Vert_{\mathcal{H}^m(\Omega)}^{2} = \sum_{i=0}^{m} \vert v
\vert_{\mathcal{H}^i(\Omega)}^{2}.
\end{equation*}

On a regular mesh of size $h$, the condition number of the finite element
equations for a second-order elliptic boundary value problem can be obtained
using {\it{inverse estimates}} \cite{Axelsson-01, Braess-07, Ciarlet-78}.
Similar inverse estimates are of interest for the isogeometric framework using
NURBS basis functions.

To keep the article self-contained, we recall some results from \cite{
BVCHughesS-06, Schumaker-07}.
\begin{theorem}
\label{th:Inverse}
Let $\mathcal{S}_h$ be the spline space consisting of piecewise polynomials of
degree $p$ associated with uniform partitions.
Then there exists a constant $C=C(shape)$, such that for all $0 \leq l \leq
m$,
\begin{equation}
\label{eq:InvEst}
\|v\|_{\mathcal{H}^m(\Omega)} \leq C h^{l-m} \|v\|_{\mathcal{H}^l(\Omega)}, \quad \forall v \in \mathcal{S}_h.                                 
\end{equation}
\end{theorem}
\noindent
The proof of the above theorem, for a particular case $m=2$ and $l=1$, is
given in \cite{BVCHughesS-06}.
More general inverse inequalities can be easily derived following the same
approach.
By taking $m=1$ and $l=0$, the following can be easily derived from
(\ref{eq:InvEst})
\begin{equation}
\label{eq:EstInv}
 a(v,v)=\int_{\Omega} |\nabla v|^2 \leq C h^{-2}\|v\|^2.
\end{equation}
Under suitable conditions the condition number related to elliptic problems in
finite element analysis scales as $h^{-2}$ \cite{Ern-91, Johnson-87,
Strang-73}.
We prove the similar result for the stiffness matrix arising in isogeometric
discretization.

We first prove
\begin{lemma}
\label{le:BilBouH}
There exist constants $C_1$ and $C_2$ independent of $h$ (but may depend on
$p$), such that for all $v = \displaystyle \sum_{i=1}^{n_h} v_i N_i \in
\mathcal{S}_h$,
\begin{equation}
\label{eq:BasisCondH}
C_1h^2\|\{v_i\}\|^2 \leq \Big\|\sum_{i=1}^{n_h} v_i N_i \Big\|^2\leq
C_2h^2\|\{v_i\}\|^2.
\end{equation}
\end{lemma}
\begin{proof}
We only consider the non-trivial case: there exists some $i$ for which $v_i
\neq 0$.
For any $Q \in \mathcal{Q}_h$, there are $(p+1)^{2}$ basis functions with
non-zero support.
Let $ \mathcal{I}_{h}^{Q} \equiv \{i^{Q}_{1}, i^{Q}_{2}, \ldots ,
i^{Q}_{p+1}\} \times \{j^{Q}_{1}, j^{Q}_{2}, \ldots , j^{Q}_{p+1}\} \subset
\{1, 2, \ldots, n_h\}$ denote the index set for the basis functions that have
non-zero support in $Q$.
Also, let $\bar{v}_q = {\displaystyle \max_{i \in \mathcal{I}_{h}^{Q} } \vert
v_i \vert}$ and $\bar{v} = {\displaystyle \max_{i= 1,2,\ldots,n_h} \vert v_i
\vert}$.
Now using positivity and partition of unity properties of basis functions, the
right hand side inequality can be proved as follows:
\begin{align*}
\| v \|^{2} = & \sum_{Q \in \mathcal{Q}_h} \int_{Q} v^{2}
= \sum_{Q \in \mathcal{Q}_h} \int_{Q}\Bigg(\sum_{i \in \mathcal{I}_{h}^{Q} }
v_i N_i \Bigg)^2
\leq \sum_{Q \in \mathcal{Q}_h} \int_{Q}\Bigg( \bar{v}_q \sum_{i \in
\mathcal{I}_{h}^{Q} } N_i \Bigg)^2 \\
\leq & \sum_{Q \in \mathcal{Q}_h} \int_{Q} \bar{v}_q^2
\leq \sum_{Q \in \mathcal{Q}_h} h_{Q}^{2} \bar{v}^2_q
\leq \sum_{Q \in \mathcal{Q}_h} h_{Q}^{2} \sum_{i \in \mathcal{I}_{h}^{Q}}
v_i^{2}\\
\leq & h^{2} \sum_{Q \in \mathcal{Q}_h} \sum_{i \in \mathcal{I}_{h}^{Q} }
v_i^{2}
\leq C_2 h^{2} \sum_{i=1}^{n_h} v_i^{2}= C_2h^2\|\{v_i\}\|^2.
\end{align*}
For the left hand side inequality,
\begin{align*}
h^2\|\{v_i\}\|^2 = & h^2 \sum_{i=1}^{n_h} v_i^2 \leq h^2 \sum_{i=1}^{n_h}
\bar{v}^2
= h^2 {n_h} \bar{v}^2 \leq h^2 \bigg( \frac{C}{h}\bigg)^2 \bar{v}^2 = C^2
\bar{v}^2 \\
= & C^2 \|\{v_i\}\|_{L_\infty}^2 \leq C^2 K_1^2 \|v\|_{L_\infty}^2
\text{ $\Big($using (\ref{eq:BasisCondP}), $ K_1^{-1}\|\{v_i\}\|_{L_\infty}
\leq \big\| \sum v_i N_i \big\|_{L_\infty}\Big)$ }\\
\leq & C^2 K_1^2 \|v\|^2.
\end{align*}
The result then follows by taking $\displaystyle C_1= \left(\displaystyle \frac{1}{C^2K_1^2}\right).$ \hfill
\end{proof}

We now turn to the problem of obtaining bounds on the extremal eigenvalues and
the condition number.
\begin{theorem}
\label{th:ConBouH}
Let $A$ be the stiffness matrix $A=(a_{ij})$, where $a_{ij}=a(N_i,N_j)=
\displaystyle \int _{\Omega} \nabla N_i \cdot \nabla N_j$.
Then the upper bound on $\lambda_{\text{max}}$ and lower bound on
$\lambda_{\text{min}}$ are given by
\begin{center}
$\lambda_{\text{max}} \leq c_1$
\quad and \quad
$\lambda_{\text{min}} \geq c_2 h^2$,
\end{center}
where $c_1, c_2$ are constants independent of $h$. 
The bound on $\kappa (A)$ is given by
\begin{center}
$\kappa (A) \leq C h^{-2}$,
\end{center}
where $C$ is a constant independent of $h$.
\end{theorem}

\begin{proof}
Let $v=\displaystyle \sum_{i=1}^{n_h} v_i N_i$.  Then $a(v,v)= \{v_i\}\cdot
A\{v_i\}$, where $\{v_i\}=\{v_1,v_2,\ldots,v_{n_h}\}$.
Using the inverse estimate (\ref{eq:EstInv}),
\begin{align*}
\frac{\{v_i\}\cdot A\{v_i\}}{\|\{v_i\}\|^2}= &
\frac{a(v,v)}{\|\{v_i\}\|^2}\leq \frac{C h^{-2} \|v\|^2}{\|\{v_i\}\|^2}.
\end{align*}
Using (\ref{eq:BasisCondH}),
\begin{align*}
&\displaystyle \frac{C h^{-2} \|v\|^2}{\|\{v_i\}\|^2} \leq \frac{C h^{-2} C_2
h^2\|\{v_i\}\|^2 }{\|\{v_i\}\|^2}=CC_2=c_1.
\end{align*}
Hence,
\begin{equation}
\label{eq:LambdaMaxBound}
\displaystyle \lambda_{\text{max}} =
\sup_{\substack{v\neq0}} \frac{\{v_i\}\cdot
A\{v_i\}}{\|\{v_i\}\|^2}\leq c_1.
\end{equation}
On the other hand, for the bounds on $\lambda_{\text{min}}$, by using
coercivity of bilinear form $a(v,v)$,
\begin{align*}
\frac{\{v_i\}\cdot A\{v_i\}}{\|\{v_i\}\|^2}= &\frac{a(v,v)}{\|\{v_i\}\|^2}\geq
\frac{\alpha \|v\|_{H^1}^2}{\|\{v_i\}\|^2}
\geq \frac{\alpha \|v\|^2}{\|\{v_i\}\|^2}.
\end{align*}
Using (\ref{eq:BasisCondH}) again,
\begin{align*}
\displaystyle \frac{\alpha \|v\|^2}{\|\{v_i\}\|^2} \ge \frac{\alpha_1 C_1
h^2\|\{v_i\}\|^2 }{\|\{v_i\}\|^2}= & \alpha_1C_1h^2=c_2h^2.
\end{align*}
Hence,
\begin{equation}
\label{eq:LambdaMinBound}
\displaystyle \lambda_{\text{min}} =
\inf_{\substack{v\neq0}} \frac{\{v_i\}\cdot A\{v_i\}}{\|\{v_i\}\|^2}\geq c_2
h^2.
\end{equation}
The condition number of the stiffness matrix is given by
\begin{equation*}
\kappa (A) =\frac{\lambda_{\text{max}}}{\lambda_{\text{min}}},
\text{where}~\lambda_{\text{max}}= \max_{\substack{v\neq0}} \frac{\{v_i\}\cdot
A\{v_i\}}{\|\{v_i\}\|^2},~
\text{and}~
\lambda_{\text{min}}= \min_{\substack{v\neq0}} \frac{\{v_i\}\cdot
A\{v_i\}}{\|\{v_i\}\|^2}.
\end{equation*}
From (\ref{eq:LambdaMaxBound}) and (\ref{eq:LambdaMinBound}),
\begin{equation}
\kappa (A) \leq C h^{-2}.
\end{equation}
\hfill\end{proof}


\subsubsection{$p$-refinement}
\label{Subsec:CondNum4StiffnessMatrix-p}

In this section for $p$-refinement, we prove upper and lower bounds for the
maximum eigenvalue, lower bounds for the minimum eigenvalue, and upper bounds
for the condition number.

Let $\mathcal{S}_p$ be the tensor product space of spline functions of
degree $p$.

The following lemma is well known generalization of a theorem of Markov due to
Hill, Szechuan and Tamarkin \cite{ Bellman-44, Olsen-95}.
\begin{lemma}[Schmidt's inequality]
There exists a constant $C$ independent of $p$ such that for any polynomial
$f(x)$ of degree $p$,
\begin{equation}
\label{eq:SchIne}
\int_{-1}^{1}(f'(x))^2 dx \leq C p^4 \int_{-1}^{1}(f(x))^2 dx.
\end{equation}
\end{lemma}
\noindent
Note: No such constant $C$ exists so that (\ref{eq:SchIne}) holds for all
$f(x)$ with the exponent smaller than 4.

Let $I=(-1,1)$.
Using (\ref{eq:SchIne}),
\begin{equation}
\label{eq:SchIneN}
\int_{I} \left(\frac{dN_p(\xi)}{d\xi}\right)^2 d\xi \leq C p^4
\int_{I}(N_p(\xi))^2 d\xi.
\end{equation}
Using (\ref{eq:SchIneN}),
\begin{equation}
\label{SchIneB}
\begin{split}
\int_{\Omega} \nabla N_p(\xi,\eta) \cdot \nabla N_p(\xi,\eta) d\xi d\eta & =
\displaystyle \int_I \int_I \left[ \left(\frac {\partial
N_p(\xi,\eta)}{\partial \xi}\right)^2+
\left(\frac {\partial N_p(\xi,\eta)}{\partial \eta}\right)^2 \right] d\xi
d\eta \\
& \leq \displaystyle C p^4 \int_{I\times I} (N_p(\xi,\eta))^2 d\xi d\eta.
\end{split}
\end{equation}
Moreover, the following estimate directly follows from Schmidt's inequality
and (\ref{SchIneB}):
\begin{equation}
\label{eq:EstInvP}
a(v,v)=\int_{\Omega} |\nabla v|^2 \leq C p^4\|v\|^2.
\end{equation}
We now have a similar result like Lemma \ref{le:BilBouH} for the
$p$-refinement.
\begin{lemma}
\label{le:BilBouP}
There exist constants $C_1$ and $C_2$ independent of $p$ such that for
all\\ $\displaystyle v =\sum_{i=1}^{n_p} v_i N_i \in \mathcal{S}_p$,
\begin{equation}
\label{eq:BasConP}
\displaystyle\frac {C_1}{(p^24^p)^2}\|\{v_i\}\|^2 \leq \displaystyle
\Big\|\sum_{i=1}^{n_p} v_i N_i \Big\|^2\leq C_2\|\{v_i\}\|^2.
\end{equation}
\end{lemma}
\begin{proof}
From the stability of B-Splines there exists a constant $\gamma$ that
depends on the degree $p$ such that
\begin{equation}
\label{eq:BSplineStab}
\displaystyle \Big\| \sum_{i=1}^{n_p} v_iN_i\Big\| \leq \|\{v_i\}\| \leq
\gamma \Big\| \sum_{i=1}^{n_p} v_iN_i\Big\|,
\end{equation}
From (\ref{eq:BspKapD}), $\gamma = p^24^p$.
In the estimate (\ref{eq:BasConP}), the right hand side inequality follows
easily from nonnegativity and the partition of unity properties of basis
functions.
The left hand side inequality follows from (\ref{eq:BSplineStab}).
\hfill\end{proof}

For the $p$-refinement of isogeometric discretization, the analog to Theorem
\ref{th:ConBouH} is
\begin{theorem}
\label{th:ConBouP}
Let $\{N_i\}$ be a set of basis functions of $\mathcal{S}_p$ on a unit square.
Then
\begin{center}
$\kappa(A)\leq C p^816^p$.
\end{center}
\end{theorem}
\begin{proof}
We prove this theorem following the same approach as for the $h$-refinement
estimates.
Let $v= \displaystyle \sum_{i=1}^{n_p} v_i N_i$, where
$\{v_i\}=\{v_1,v_2,\ldots,v_{n_p}\}$.
Now using (\ref{eq:EstInvP}) and (\ref{eq:BasConP}),
\begin{equation*}
\begin{split}
\frac{\{v_i\}\cdot A\{v_i\}}{\|\{v_i\}\|^2}= \frac{a(v,v)}{\|\{v_i\}\|^2}\leq
\frac{C p^4 \|v\|^2}{\|\{v_i\}\|^2}
\leq \frac{C p^4 C_2\|\{v_i\}\|^2 }{\|\{v_i\}\|^2}
= C C_2 p^4 = C p^4.
\end{split}
\end{equation*}
Hence,
\begin{equation}
\label{eq:LambdaMaxBoundP}
\lambda_{\text{max}}= \max_{\substack{v\neq0}} \frac{\{v_i\}\cdot
A\{v_i\}}{\|\{v_i\}\|^2} \leq C p^4.
\end{equation}
To prove the lower bound for $\lambda_{\text{min}}$ we use (\ref{eq:BasConP})
and coercivity of bilinear form,
\begin{equation*}
\begin{split}
\frac{\{v_i\}\cdot A\{v_i\}}{\|\{v_i\}\|^2}&= \frac{a(v,v)}{\|\{v_i\}\|^2}\geq
\frac{\alpha\|v\|_{H^1}^2}{\|\{v_i\}\|^2} \ge
\frac{\alpha\|v\|^2}{\|\{v_i\}\|^2}\\
&\geq  \frac{\alpha \displaystyle \frac {C_1}{(p^24^p)^2}\|\{v_i\}\|^2}{ \|\{v_i\}\|^2}
= {\frac {\alpha C_1}{(p^24^p)^2}}={\frac {C}{(p^416^p)}}.
\end{split}
\end{equation*}
Hence,
\begin{equation}
\label{eq:LambdaMinBoundP}
\lambda_{\text{min}}= \min_{\substack{v\neq0}} \frac{\{v_i\}\cdot
A\{v_i\}}{\|\{v_i\}\|^2} \geq \frac {C}{(p^416^p)}.
\end{equation}
From (\ref{eq:LambdaMaxBoundP}) and (\ref{eq:LambdaMinBoundP}),
\begin{align*}
\kappa (A) =\frac{\lambda_{\text{max}}}{\lambda_{\text{min}}} \leq
\frac{Cp^4}{{\left( \displaystyle\frac {C}{(p^416^p)}\right)}}\leq C(p^816^p).
\end{align*}
\hfill\end{proof}

\begin{remark}
Theorem~\ref{th:ConBouP} can be easily generalized for higher dimensions.
The bound for the condition number of the stiffness matrix for a
$d$-dimensional problem is given by $(p^{4+2d}4^{pd})$.
\end{remark}

While we proved an upper bound on the maximum eigenvalue of the stiffness
matrix using B-Spline basis functions, Theorem~\ref{th:ConBouP} is independent
of the choice of the basis functions (it holds for all kind of basis
functions, not just spline functions).
From numerical experiments using B-Spline basis functions (see Table
\ref{tab:maxeig_p_h}), we observe that $\lambda_{\text{max}}$ depends linearly on the polynomial degree $p$, which motivates further investigations.

\begin{table}[!b]
\caption{ Maximum eigenvalue of the stiffness matrix $A$}
\label{tab:maxeig_p_h}
\begin{center}
\begin{tabular}{|c|c|c|c|c|c|c|c|c|}\hline
\backslashbox{$p$}{$h^{-1}$} & 1 & 2 & 4 & 8 & 16 & 32 & 64 & 128 \\ \hline
2 & 0.36 & 1.42  & 1.42 & 1.49 & 1.50 & 1.50 & 1.50 & 1.50 \\\hline
3 & 0.45 & 1.04 & 1.37 & 1.52 & 1.56 & 1.57 & 1.57 & 1.57 \\\hline
4 & 0.41 &  0.94 &  1.33 &  1.72 &  1.81 &  1.83 & 1.84 & 1.84 \\\hline
5 & 0.35 &  0.88 &  1.32 &  1.93 &  2.10 &  2.14 & 2.14 & 2.14  \\\hline
6 &  0.34 &  0.85 &  1.32 & 2.12 &  2.40 &  2.46 &  2.47 &  2.47\\\hline
7  & 0.33 & 0.84 &  1.32 & 2.26 & 2.70 &  2.78 & 2.80 & 2.80\\\hline
8 &  0.32 & 0.83 &  1.33 & 2.36 & 2.99 &  3.11 & 3.13 & 3.14 \\\hline
9 & 0.31 &  0.82 & 1.33 & 2.43 &  3.29 & 3.44 &  3.47 & 3.47 \\\hline
10 & 0.31 &  0.82 & 1.34 &  2.47 &  3.56 &  3.77 &  3.80 & 3.81 \\\hline
20 & 0.29 &  0.78 & 1.36 & 2.65 & 5.02 & 6.95 & 7.20 & 7.23 \\\hline
30 & 0.29 &  0.78 & 1.36 & 2.69 & 5.28 & 9.38 & 10.55 & 10.66\\\hline
\end{tabular}
\end{center}
\end{table}

The lower bound on the minimum eigenvalue depends on the stability of the
B-Spline basis functions, which cannot be improved further (especially beyond
the de Boor's conjecture).
On the other hand, the upper bound on the maximum eigenvalue directly depends
on the upper bound of the bilinear form $a(v,v)$.
We can improve the bound for $a(v,v)$ given in (\ref{eq:EstInvP}).
In the following theorem we improve this bound and provide our main
result.
\begin{theorem}
\label{Thm:StippImproved-p}
For the two dimensional problem the improved upper bound for the condition
number of the stiffness matrix $A$ is given by
\begin{equation}
 \kappa(A)\leq  C p^2(p^24^p)^2= C p^616^p.
\end{equation}
The bound for a $d$-dimensional problem is given by
\begin{equation}
\kappa(A)\leq C p^{2d+2}4^{pd}.
\end{equation}
\end{theorem}
\noindent
For the sake of clarity we will give the proof of
Theorem~\ref{Thm:StippImproved-p} in parts in Lemmas~\ref{le:UBDiagA}-\ref{Lemma:MaxEV=Const}.

It is clear from Table~\ref{tab:maxeig_p_h} that the maximum eigenvalue of the
stiffness matrix is independent of $p$ for the coarsest mesh size $h=1$, and linearly dependent of $p$ asymptotically.
In the analysis, we consider two dimensional problem on the coarsest mesh first and extend it to finer meshes later.
On the coarsest mesh we have B-Spline basis functions of degree $p$ in one variable $\xi$,
\begin{equation*}
N^p_{i,\xi}=(-1)^i {p \choose i} (\xi-1)^{p-i}\xi^i, \quad i=0,1,2,\ldots,p.
\end{equation*}
Similarly in variable $\eta$,
\begin{equation*}
N^p_{j,\eta}=(-1)^j {p \choose j} (\eta-1)^{p-j}\eta^j, \quad
j=0,1,2,\ldots,p.
\end{equation*}
Two variable B-Spline basis functions on the coarsest mesh is given by the
tensor product
\begin{equation*}
N^{p,p}_{i,j,\xi,\eta}=(-1)^{i+j} {p \choose i} {p \choose j} \xi^i \eta^j
(\xi-1)^{p-i} (\eta-1)^{p-j},\quad i,j=0,1,2,\ldots,p.
\end{equation*}

We construct an upper bound of the diagonal entries of the stiffness
matrix on the coarsest mess i.e. single element stiffness matrix $A^{e}$.
\begin{lemma}
\label{le:UBDiagA}
There exists a constant $C$ independent of $p$, such that
\begin{equation}
\label{eq:Bound_on_a_ii}
A^{e}_{(i,j),(i,j)}=a(N^{p,p}_{i,j,\xi,\eta},N^{p,p}_{i,j,\xi,\eta})= \int_0^1
\int_0^1 \nabla N^{p,p}_{i,j,\xi,\eta} \cdot \nabla N^{p,p}_{i,j,\xi,\eta}d\xi
d\eta \le C.
\end{equation}
\end{lemma}
\begin{proof}
We provide the major points of the proof.
Some of the details can be found in the research report \cite{Gahalaut-14}.
For all $i,j=0,1,2,\ldots,p$,
\begin{equation*}
\begin{split}
& a(N^{p,p}_{i,j,\xi,\eta},N^{p,p}_{i,j,\xi,\eta})= \int_0^1 \int_0^1 \nabla
N^{p,p}_{i,j,\xi,\eta} \cdot \nabla N^{p,p}_{i,j,\xi,\eta}d\xi d\eta \\
& = {p \choose i}^2 {p \choose j} ^2 \int_0^1 \int_0^1 \big\{i \xi^{i-1}
\eta^j (\xi-1)^{p-i} (\eta-1)^{p-j} +\\
& \hspace{5cm} (p-i) \xi^{i} \eta^j (\xi-1)^{p-i-1} (\eta-1)^{p-j}\big\}^2
d\xi d\eta \\
& + {p \choose i}^2 {p \choose j} ^2 \int_0^1 \int_0^1 \big\{j \xi^{i}
\eta^{j-1} (\xi-1)^{p-i} (\eta-1)^{p-j} + \\
& \hspace{5cm} (p-j) \xi^{i} \eta^j (\xi-1)^{p-i} (\eta-1)^{p-j-1}\big\}^2
d\xi d\eta\\
& \equiv I+I\hspace{-1mm}I.
\end{split}
\end{equation*}
Now,
\begin{equation*}
\begin{split}
I & = {p \choose i} ^2 {p \choose j} ^2 \int_0^1 \int_0^1 \left (i^2
\xi^{2(i-1)} \eta^{2j} (\xi-1)^{2(p-i)} (\eta-1)^{2(p-j)} \right) d\xi d\eta
+\\
& \quad {p \choose i} ^2 {p \choose j} ^2 \int_0^1 \int_0^1 \left ((p-i)^2
\xi^{2i} \eta^{2j} (\xi-1)^{2(p-i-1)} (\eta-1)^{2(p-j)}\right)d\xi d\eta+\\
& \quad {p \choose i} ^2 {p \choose j} ^2 \int_0^1 \int_0^1 \left (2i(p-i)
\xi^{2i-1} \eta^{2j} (\xi-1)^{2p-2i-1} (\eta-1)^{2(p-j)}\right)d\xi d\eta\\
& \equiv I_1+I_{2}+I_{3}.
\end{split}
\end{equation*}
After simplifying (using results on factorial functions), we get
\begin{equation*}
\begin{split}
I_1 & \leq \frac{1}{2}  {p \choose i} ^2
{p \choose j} ^2 \frac{(2i)!  (2p-2i)!}{(2p-1)!} \frac{(2j)!
(2p-2j)!}{(2p+1)!},\\
I_2 & \leq \frac{1}{2}  {p \choose i} ^2
{p \choose j} ^2 \frac{(2i)!  (2p-2i)!}{(2p-1)!} \frac{(2j)!
(2p-2j)!}{(2p+1)!},\\
I_3 &= - \frac{1}{2}  {p \choose i} ^2
{p \choose j} ^2 \frac{(2i)!  (2p-2i)!}{(2p-1)!} \frac{(2j)!
(2p-2j)!}{(2p+1)!}.
\end{split}
\end{equation*}
For all $i=0,1,2,\ldots,p$,
\begin{equation*}
\begin{split}
I & = I_1+I_2+I_3 \\
& = \begin{cases} I_2, & {\text{ if } i=0,}
\\
I_1+I_2+I_3, & {\text{ if } i=1,2,\ldots,p-1,}
\\
I_1, & {\text{ if } i=p,}
\end{cases}\\
& \leq \left\{ {p \choose i} ^2 \frac{(2i)!  (2p-2i)!}{(2p)!} \right \}
\left \{ {p \choose j} ^2 \frac{(2j)!  (2p-2j)!}{(2p)!} \right\}= I_aI_b,
\text{ where} \\
I_a & = {p \choose i} ^2 \frac{(2i)!  (2p-2i)!}{(2p)!} = \frac{p!
p!}{i!i!(p-i)!(p-i)!} \frac{(2i)!  (2p-2i)!}{(2p)!},\\
I_b & = {p \choose j} ^2 \frac{(2j)!  (2p-2j)!}{(2p)!}= \frac{p!
p!}{j!j!(p-j)!(p-j)!} \frac{(2j)!  (2p-2j)!}{(2p)!}.
\end{split}
\end{equation*}
We prove that $I_a \leq C$ by induction on $p$, where $C$ is a constant
independent of $p$.
For $p=1$, we have $I_{a}=1\text{ for all } i=0,1$.
Hence, the result holds for the base case.
Assume that the result holds for $p=m \text{ and for all } i=0,1,2,\ldots,m$,
\begin{equation}
\label{eq:induction}
\displaystyle \frac{m!  m!}{i!i!(m-i)!(m-i)!} \frac{(2i)!  (2m-2i)!}{(2m)!}
\leq C.
\end{equation}
Now we show that the result holds for $p=m+1~\text{and for all}~
i=0,1,2,\ldots,m+1$.
We have
\begin{equation*}
\begin{split}
& \frac{(m+1)!  (m+1)!}{i!i!(m+1-i)!(m+1-i)!} \frac{(2i)!
(2(m+1)-2i)!}{(2(m+1))!} \\
& = \begin{cases}\displaystyle \left( \frac{m^2+2m+1}{4m^2+6m+2}\right) \left(
\frac{4(m-i)^2+6(m-i)+2}{(m-i)^2+2(m-i)+1}\right)
\left\{ \frac{m!  m!}{i!i!(m-i)!(m-i)!} \frac{(2i)!  (2m-2i)!}{(2m)!}\right\},
\\ \hfill{\text{if } i=0,1,2,\ldots,m,}
\\
1, {\text{if } i=m+1.}
\end{cases}
\end{split}
\end{equation*}
Using (\ref{eq:induction}) and since $\displaystyle
\left(\frac{m^2+2m+1}{4m^2+6m+2}\right)\left(\frac{4(m-i)^2+6(m-i)+2}{(m-i)^2+2(m-i)+1}\right)\le1$,
we get for all $i= 0,1,2,\ldots,m+1$,
\begin{equation*}
 \frac{(m+1)! (m+1)!}{i!i!(m+1-i)!(m+1-i)!} \frac{(2i)! (2(m+1)-2i)!}{(2(m+1))!} \leq C.
\end{equation*}
We now have $I_a \leq C$, where $C$ is a constant independent of $p$.
Similarly we can obtain that $I_b \leq C$.
Hence,
\begin{center}
$I=I_aI_b \leq C$.
\end{center}
Proceeding in the same way for ${I\hspace{-1mm}I}$, we can prove that
\begin{equation*}
{I\hspace{-1mm}I} \leq C.
\end{equation*}
Finally,
\begin{equation*}
\displaystyle a(N^{p,p}_{i,j,\xi,\eta},N^{p,p}_{i,j,\xi,\eta}) = I +
{I\hspace{-1mm}I} \leq C.
\end{equation*}
\hfill\end{proof}
 
Thus, we have proved that $a(N^{p,p}_{i,j,\xi,\eta},N^{p,p}_{i,j,\xi,\eta})$
is bounded by a constant independent of $p$.
Since the upper bound of the diagonal entries is the upper bound of all the
entries of the stiffness matrix, the maximum entry of the stiffness matrix is
bounded by a constant independent of $p$,
\begin{equation}
\label{eq:bound_aii_2d}
\displaystyle a(N^{p,q}_{i,j,\xi,\eta},N^{p,q}_{k,l,\xi,\eta})\leq C.
\end{equation}
Similarly, we can prove for three dimensional problem that
\begin{equation}
\label{eq:bound_aii_3d}
\displaystyle
a(N^{p,q,r}_{i,j,k,\xi,\eta,\zeta},N^{p,q,r}_{l,m,n,\xi,\eta,\zeta})\leq C.
\end{equation}
Using (\ref{eq:bound_aii_2d}) and (\ref{eq:bound_aii_3d}) we have
\begin{lemma}
\label{le:LBlambdamax}
The maximum eigenvalue of the element stiffness matrix $A^{e}$ can be bounded below by a
constant $C$ independent of $p$,
\begin{equation*}
\lambda_{\text{max}}(A^{e}) \geq C.
\end{equation*}
\end{lemma}
\begin{proof}
We prove this by using the basics of matrix norms.
The max-norm of a matrix is the element-wise norm defined by
\begin{equation*}
\|A^{e}\|_{\text{max}} = \ \text{max} \{|a_{ij}|\}.
\end{equation*}
From (\ref{eq:Bound_on_a_ii}),
\begin{equation*}
\text{max} \{|a_{ij}|\}=C,
\end{equation*}
where $C$ is independent of $p$.
By the equivalence of norms we have
\begin{equation*}
\|A^{e}\|_2 \ge \|A^{e}\|_{\text{max}} = C.
\end{equation*}
Hence,
\begin{equation*}
\lambda_{\text{max}}(A^{e}) \ge C.
\end{equation*}
\hfill\end{proof}
 
To bound $\lambda_{\text{max}}$ from above we bound the spectral norm by the
$\ell_1$-norm in
\begin{lemma}
\label{Lemma:BoundEVmax-l1-p}
For any fixed $k$ and $l$ such that $0 \le k, l \le p$ and for any $0 \le i,j
\le p$,
\begin{center}
$\displaystyle \sum_{i=0}^{p}
\sum_{j=0}^{p}\big|a(N^{p,p}_{k,l,\xi,\eta},N^{p,p}_{i,j,\xi,\eta})\big| < C$,
\end{center}
where $C$ is a constant independent of $p$.
\end{lemma}
\begin{proof}
We again provide the main steps and for details refer the reader to the research report
\cite{Gahalaut-14}.
We have
\begin{equation*}
N^{p,p}_{0,0,\xi,\eta}=(1-\xi)^{p} (1-\eta)^{p}.
\end{equation*}
We first prove 
\begin{equation*}
\displaystyle \sum_{i=0}^{p}
\sum_{j=0}^{p}\big|a(N^{p,p}_{0,0,\xi,\eta},N^{p,p}_{i,j,\xi,\eta})\big| < C,
\end{equation*}
where $C$ is a constant independent of $p$.
We have
\begin{equation*}
\begin{split}
& a(N^{p,p}_{0,0,\xi,\eta},N^{p,p}_{i,j,\xi,\eta})= \int_0^1 \int_0^1 \nabla
N^{p,p}_{0,0,\xi,\eta} \cdot \nabla N^{p,p}_{i,j,\xi,\eta}d\xi d\eta\\
& = \int_0^1 \int_0^1 \left(\frac{ \partial }{\partial
\xi}N^{p,p}_{0,0,\xi,\eta}\frac{ \partial }{\partial
\xi}N^{p,p}_{i,j,\xi,\eta}\right) d\xi d\eta
+ \int_0^1 \int_0^1 \left(\frac{ \partial }{\partial
\eta}N^{p,p}_{0,0,\xi,\eta}\frac{ \partial }{\partial
\eta}N^{p,p}_{i,j,\xi,\eta}\right) d\xi d\eta \\
& = I+I\hspace{-1mm}I.
\end{split}
\end{equation*}
Now,
\begin{equation*}
\begin{split}
I & = -p {p \choose i}
{p \choose j} \left( \int_0^1 \eta^j (1-\eta)^{2p-j} d\eta \right) \\
& \quad \quad \quad \quad \left( \int_0^1i \xi^{i-1} (1-\xi)^{2p-i-1} d\xi -
\int_0^1(p-i) \xi^{i} (1-\xi)^{2p-i-2} d\xi \right)\\
& I = \begin{cases} \displaystyle {p \choose j}
\frac{p^2}{(4p^2-1)}\frac{(j)!(2p-j)!}{(2p)!} , \hfill \text{if } i=0, \\
\displaystyle -p {p \choose i} {p \choose j}\frac{2p}{(2p+1)}
\frac{(i)!(2p-i)!}{(2p)!}
\frac{(j)!(2p-j)!}{(2p)!} \frac{1}{(2p-i)} \left(1 - \frac{(p-i)}{(2p-i-1)}
\right),\\ \hfill \text{if } i =1,2,\ldots,p-1,\\
\displaystyle - {p \choose j} \frac{2p}{(2p+1)} \frac{(p)!(p)!}{(2p)!}
\frac{(j)!(2p-j)!}{(2p)!}, \hfill \text{if } i = p.
\end{cases}
\end{split}
\end{equation*}
A similar expression can be obtained for $I\hspace{-1mm}I$.
We want to calculate
\begin{center}
$\displaystyle
\sum_{i=0}^{p}\sum_{j=0}^{p}\big|a(N^{p,p}_{0,0,\xi,\eta},N^{p,p}_{i,j,\xi,\eta})\big|$.
\end{center}
For $i = 0$,
\begin{equation*}
\begin{split}
\displaystyle \sum_{j=0}^{p}|I| & = \sum_{j=0}^{p} {p \choose j} \displaystyle
\frac{p^2}{(4p^2-1)}\frac{(j)!(2p-j)!}{(2p)!}
< \frac{1}{3} \sum_{j=0}^{p} {p \choose j} \displaystyle
\frac{(j)!(2p-j)!}{(2p)!}\\
& = \frac{1}{3} \sum_{j=0}^{p} \displaystyle \frac{(p)!(2p-j)!}{(2p)!(p-j)!} <
\frac{1}{3}\left( 1+\frac{1}{2}+\frac{1}{4}+\frac{1}{8}+\ldots+\frac{p!p!}{(2p)!}
\right) <1.
\end{split}
\end{equation*}
For $i =1,2,\ldots,p-1$,
\begin{equation*}
\begin{split}
& \displaystyle \sum_{i=1}^{p-1} \sum_{j=0}^{p}|I| \\
& = \sum_{i=1}^{p-1} \sum_{j=0}^{p} {p \choose i} {p \choose j} \displaystyle
\frac{2p^2}{(2p+1)} \frac{(i)!(2p-i)!}{(2p)!}
\frac{(j)!(2p-j)!}{(2p)!} \frac{1}{(2p-i)} \left(1 - \frac{(p-i)}{(2p-i-1)}
\right)\\
& <\frac{1}{2} \displaystyle \sum_{i=1}^{p-1}
\frac{p!(2p-i-2)!}{(2p-2)!(p-i)!}
< \frac{1}{2}\left(
\frac{1}{2}+\frac{1}{4}+\frac{1}{8}+\ldots+\frac{p!(p-1)!}{(2p-2)!} \right)< 1.
\end{split}
\end{equation*}
For $i =p$,
\begin{equation*}
\begin{split}
\displaystyle \sum_{j=0}^{p}|I| & = \sum_{j=0}^{p} {p \choose j} \displaystyle
\frac{2p}{(2p+1)} \frac{(p)!(p)!}{(2p)!} \frac{(j)!(2p-j)!}{(2p)!}
< \sum_{j=0}^{p} \displaystyle \frac{p!p!p!(2p-j)!}{(2p)!(2p)!(p-j)!}\\
& <\left( \frac{1}{2}+\frac{1}{4}+\frac{1}{8}+\ldots+\frac{(p!)^4}{((2p)!)^2}
\right) <1.
\end{split}
\end{equation*}
Hence,
\begin{equation}
\label{eq:rsum1}
\displaystyle \sum_{i=0}^{p} \sum_{j=0}^{p}|I|< C,
\end{equation}
where $C$ is independent of $p$.
Similarly, we have
\begin{equation}
\label{eq:rsum2}
\displaystyle \sum_{i=0}^{p} \sum_{j=0}^{p}|I\hspace{-1mm}I|< C.
\end{equation}
Therefore, from (\ref{eq:rsum1}) and (\ref{eq:rsum2}),
\begin{equation}
\label{eq:rsum}
\displaystyle \sum_{i=0}^{p}
\sum_{j=0}^{p}\big|a(N^{p,p}_{0,0,\xi,\eta},N^{p,p}_{i,j,\xi,\eta})\big| < C,
\end{equation}
where $C$ is a constant independent of $p$.

We have bounded by a constant the absolute row sum for the first row of the
element stiffness matrix.
Since on a uniform mesh the absolute row sum for all rows of the  element stiffness
matrix are of the same order upto a constant, we get the desired result:
\begin{center}
$\displaystyle \sum_{i=0}^{p}
\sum_{j=0}^{p}\big|a(N^{p,p}_{k,l,\xi,\eta},N^{p,p}_{i,j,\xi,\eta})\big| < C$,
\end{center}
for any fixed $k$ and $l$ such that $0 \le k, l \le p$ and for any $0 \le i,j
\le p$, where $C$ is a constant independent of $p$.
\hfill\end{proof}

Similar results can be obtained for higher dimensions.
The next lemma is a direct consequence of the Lemma~\ref{Lemma:BoundEVmax-l1-p}
and gives an upper bound for the maximum eigenvalue.
\begin{lemma}
\label{le:UBlambdamax}
The maximum eigenvalue of the  element stiffness matrix $A^{e}$ can be bounded above by a
constant $C$ that is independent of $p$:
\begin{equation*}
\lambda_{\text{max}}(A^{e}) \le C.
\end{equation*}
\end{lemma}
\begin{proof}
We have
\begin{center}
$\displaystyle \sum_{i=0}^{p}
\sum_{j=0}^{p}\big|a(N^{p,p}_{k,l,\xi,\eta},N^{p,p}_{i,j,\xi,\eta})\big| < C$,
\end{center}
where $C$ is a constant independent of $p$, which implies $\|A^{e}\|_{1} \le C $.
Since $A^e$ is a symmetric matrix, we have $\|A^e\|_1 = \|A^e\|_{\infty}$. Therefore, using $\|A^e\|_2 \le \sqrt{\|A^e\|_1  \|A^e\|_{\infty}}$, we get
\begin{center}
$\|A^{e}\|_2\le|A^{e}\|_1 \le C$.
\end{center}
Thus,
\begin{center}
$\lambda_{\text{max}}(A^{e}) \le C.$
\end{center}
\hfill\end{proof}

From Lemmas~\ref{le:LBlambdamax} and \ref{le:UBlambdamax}, for the element stiffness matrix we have
\begin{lemma}
\label{Lemma:MaxEV=Const}
$\lambda_{\text{max}}(A^{e})= C$, where $C$ is a constant independent of $p$.
\end{lemma}

The results in Lemma \ref{le:LBlambdamax} and Lemma \ref{le:UBlambdamax} are proved for an element stiffness matrix on a single element mesh. Obviously these results hold for all element stiffness matrices on finer meshes. Therefore, Lemma \ref{Lemma:MaxEV=Const} holds for all element stiffness matrices on refined meshes.

Now we bound the spectral norm (maximum eigenvalue) of the global stiffness matrix by its $\ell_1$-norm, i.e., the maximum of the row sum over all the rows of the global stiffness matrix. The $\ell_1$-norm of global stiffness matrix will depend on the $\ell_1$ norm of the element stiffness matrices and on their assembly.  Therefore, the bounds for maximum eigenvalue of the global stiffness matrix can be expressed in terms of the maximum eigenvalues of the corresponding element stiffness matrices and the maximum number of overlaps  within the rows and columns of element stiffness matrices (which is the number of  element stiffness matrices that contributes at a particular nonzero position in the global matrix). 

In the process of assembling the global stiffness matrix, the overlaps within the element stiffness matrices depend on the regularity of the basis functions used in the discretization.
For $C^0$- and $C^{p-1}$-continuous basis fucntion the overlaps within the elements will be minimum and maximum, respectively. It is easy to see that  for $C^{p-1}$-continuous basis functions the overlaps will be in $(p+1)^2$ knot spans (e.g., see Fig. 1). 

Using the bound for maximum eigenvalue of element stiffness matrices we state the following lemma for the bound for maximum eigenvalue of global stiffness matrix.

\begin{lemma}
\label{le:Glambdamax}
The maximum eigenvalue of the global stiffness matrix
$\lambda_{\text{max}}(A)= C p^2$, where $C$ is a constant independent of $p$.
\end{lemma}
\begin{proof}
In the assembly of the element stiffness matrices, the maximum number of overlaps for a particular nonzero position in the global matrix is $(p+1)^2$. We have
\begin{equation*}
\begin{split}
\|A\|_{1} & \le (\text{maximum number of overlaps}) \times \|A^e\|_{1},\\
& \le  ((p+1)^2) \times C.
\end{split}
\end{equation*}
The inequality $\|A\|_2\le\sqrt{\|A\|_1\|A\|_\infty}$ and symmetry of $A$ imply
\begin{center}
$\|A\|_2\le\|A\|_1 \le C (p+1)^2$.
\end{center}
Hence,
\begin{center}
$\lambda_{\text{max}} (A)  \le Cp^2.$
\end{center}
\hfill\end{proof}

Now, using the bound for maximum eigenvalue given in Lemma~\ref{le:Glambdamax},
the proof of Theorem~\ref{Thm:StippImproved-p} follows directly.

\begin{remark}
\label{re:lambdaconj}
The estimate for maximum eigenvalue given in Lemma \ref{le:Glambdamax} is not sharp. In reality this estimate is not quadratic in $p$.  This can be explained by the following observation. In the overlapping all of the elements of element stiffness matrices are not summed in absolute value. Some of the negative entries overlap with positive entries which in result reduces the row sum of the global stiffness matrix. From our numerical experiments, we conjecture the following.  

The maximum eigenvalue of the global stiffness matrix
$\lambda_{\text{max}}(A)= C p$, where $C$ is a constant independent of $p$.
\end{remark}

\begin{figure}[h!]
  \caption{Global stiffness matrix: Assembly of element stiffness matrices for $p=4$ on $8 \times 8$ spans}
  \centering
   \includegraphics[width=0.8\textwidth,natwidth=610,natheight=642]{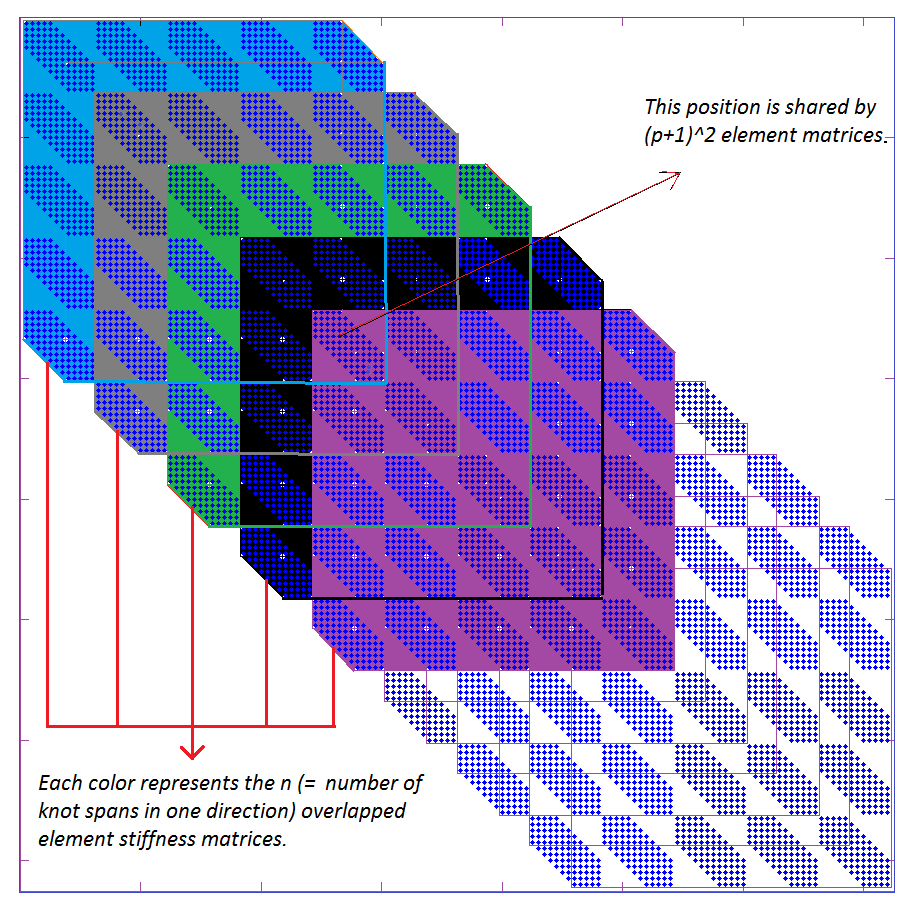}
\end{figure}

\begin{remark}
\label{re:dbc}
We used the condition number of B-Splines $\kappa \sim p2^p$ and $\lambda_{max} \sim p^2$ in reaching
the above estimates.
If we use the de Boor's conjecture (the condition number of B-Splines $\kappa
\sim 2^p$) and $\lambda_{max} \sim p$ (see Remark \ref{re:lambdaconj})  instead, then the upper bound of the stiffness matrix can be
further improved and given by
\begin{equation}
\kappa(A)\leq C p 4^{pd}.
\end{equation}
\end{remark}

\subsection{Mass matrix}
\label{Subsec:CondNum4MassMatrix}

In this section, we give estimates for the condition number of the mass
matrix with estimates for $h$-refinement in
Section~\ref{Subsubec:CondNum4MassMatrix-h} and for $p$-refinement in
Section~\ref{Subsubec:CondNum4MassMatrix-p}.


\subsubsection{$h$-refinement}
\label{Subsubec:CondNum4MassMatrix-h}
Let $M=(m_{ij})$ be the mass matrix, where
\begin{center}
$m_{ij}=(N_i,N_j)=\displaystyle \int _{\Omega} N_i N_j \quad \quad
i,j=1,2,\ldots,n_h$.
\end{center}
The following lemma gives estimates for the maximum and minimum eigenvalues of
the mass matrix with respect to $h$.
\begin{lemma}
For the extremal eigenvalues of the mass matrix $M=(m_{ij})=(N_i,N_j)$,
\begin{equation*}
C_1 h^2 \leq \lambda_{\text{min}}\leq \lambda_{\text{max}} \leq C_2 h^2,
\end{equation*}
where $C_1, C_2$ are constants independent of $h$.
Furthermore,
\begin{equation*}
c_1 \leq \kappa(M) \leq c_2,
\end{equation*}
where $c_1, c_2$ are constants independent of $h$.
\end{lemma}
\begin{proof}
Using (\ref{eq:BasisCondH}), we bound both the extremal eigenvalues of the
mass matrix.
For the minimum eigenvalue,
\begin{equation*}
\frac{\{v_i\}\cdot M\{v_i\}}{\|\{v_i\}\|^2}=  \frac{(v,v)}{\|\{v_i\}\|^2}\geq
\frac{C_1 h^2 \|\{v\}\|^2}{\|\{v_i\}\|^2} = C_1 h^2.
\end{equation*}
For the maximum eigenvalue,
\begin{equation*}
\frac{\{v_i\}\cdot M\{v_i\}}{\|\{v_i\}\|^2}=  \frac{(v,v)}{\|\{v_i\}\|^2}\leq
\frac{C_2 h^2 \|\{v\}\|^2}{\|\{v_i\}\|^2} = C_2 h^2.
\end{equation*}
So,
\begin{equation*}
C_1 h^2  \leq \lambda_{\text{min}}\leq \lambda_{\text{max}} \leq C_2 h^2.
\end{equation*}
Hence,
\begin{equation*}
c_1  \leq  \kappa(M) \leq c_2.
\end{equation*}
\hfill\end{proof}


\subsubsection{$p$-refinement}
\label{Subsubec:CondNum4MassMatrix-p}

In this section, we estimate the bounds on the extremal eigenvalues and the
condition number of the mass matrices for $p$-refinement.

\begin{lemma}
\label{le:MaPoMijBound}
The element mass matrix is a positive matrix and all of the entries of the element mass matrix
are bounded above by $\displaystyle \frac{C}{(2p+1)^2}$, where $C$ is a
constant independent of $p$.
\end{lemma}
\begin{proof}
We have
\begin{equation*}
\begin{split}
M^{e}_{(i,j),(k,l)} & = (N^{p,p}_{i,j,\xi,\eta},N^{p,p}_{k,l,\xi,\eta}) =
\int_0^1 \int_0^1 N^{p,p}_{i,j,\xi,\eta} \cdot N^{p,p}_{k,l,\xi,\eta}d\xi
d\eta \\
& = \int_0^1 \int_0^1 \left( (-1)^{i+j} {p \choose i} {p \choose j} \xi^i
\eta^j (\xi-1)^{p-i} (\eta-1)^{p-j} \right) \\
& \quad \quad \quad \left( (-1)^{k+l} {p \choose k} {p \choose l} \xi^k \eta^l
(\xi-1)^{p-k} (\eta-1)^{p-l} \right)d\xi d\eta\\
& = \left( I \right)\left( I\hspace{-1mm}I \right),
\end{split}
\end{equation*}
where
\begin{equation*}
\begin{split}
I & = {p \choose i} {p \choose k}
\left( \int_0^1 \xi^{(i+k+1)-1} (1-\xi)^{(2p-i-k+1)-1} d\xi d\eta \right) \\
& = \frac{ p!p!}{i!k!(p-i)!(p-k)!}\frac{(i+k)!(2p-i-k)!}{(2p+1)!}\\
& =\frac{1}{2p+1}{\left\{ \frac{
p!p!}{i!k!(p-i)!(p-k)!}\frac{(i+k)!(2p-i-k)!}{(2p)!}\right\}}
=\frac{1}{2p+1} I_1,
\end{split}
\end{equation*}
and
\begin{equation*}
\begin{split}
I\hspace{-1mm}I & = {p \choose j} {p \choose l} \left( \int_0^1
\eta^{(j+l+1)-1} (1-\eta)^{(2p-j-l+1)-1}d\xi d\eta \right) \\
& =\frac{ p!p!}{j!l!(p-j)!(p-l)!}\frac{(j+l)!(2p-j-l)!}{(2p+1)!}\\
& =\frac{1}{2p+1}{\left\{ \frac{
p!p!}{j!l!(p-j)!(p-l)!}\frac{(j+l)!(2p-j-l)!}{(2p)!}
\right\}} = \frac{1}{2p+1} I\hspace{-1mm}I_1.
\end{split}
\end{equation*}
By induction on $p$ we easily obtain that (as we proved in Lemma
\ref{le:UBDiagA}),
\begin{equation*}
\begin{split}
I_1 =\left\{ \frac{
p!p!}{i!k!(p-i)!(p-k)!}\frac{(i+k)!(2p-i-k)!}{(2p)!}\right\} \le C.
\end{split}
\end{equation*}
Similarly, $I\hspace{-1mm}I_1 \le C$.
Therefore
\begin{equation}
\label{eq:Upperbound_mij}
M^{e}_{(i,j),(k,l)} \leq \frac{C}{(2p+1)^2}.
\end{equation}
It is also clear that for all $p\geq1$ and $i,k=0,1,2,\ldots,p$, $I_1 > 0$,
and $I\hspace{-1mm}I_1 > 0$.
Hence, the mass matrix $M^{e}_{(i,j),(k,l)}$ is a positive matrix.
\hfill\end{proof}

\begin{lemma}
\label{le:LBlambdamaxM}
The maximum eigenvalue of the element mass matrix $M^e$ can be bounded below by
\begin{equation*}
\lambda_{\text{max}}(M^{e}) \geq \frac{C}{(2p+1)^2}.
\end{equation*}
\end{lemma}
\begin{proof}
Following the proof of Lemma \ref{le:LBlambdamax} and
(\ref{eq:Upperbound_mij}) we get the desired result.
\end{proof}

To bound $\lambda_{\text{max}}$ from above we bound the spectral norm by the
$\ell_1$-norm of the mass matrix.
In the following lemma we first compute the $\ell_1$-norm of the mass matrix.
\begin{lemma}
\label{le:LZidea}
For the mass matrix $M^{e}$ on the coarsest mesh,
\begin{equation*}
\|M^{e}\|_{1} = \frac{1}{(p+1)^2}.
\end{equation*}
\end{lemma}
\begin{proof}
We have
\begin{equation*}
\begin{split}
\displaystyle \|M\|_{1} & = \max_{\substack{i,j}} \sum_{k,l} \int_0^1 \int_0^1
N_{i,j,\xi,\eta}^{p,p}\cdot N_{k,l,\xi,\eta}^{p,p} d\xi d\eta
= \max_{\substack{i,j}} \sum_{k,l}
(N_{i,j,\xi,\eta}^{p,p},N_{k,l,\xi,\eta}^{p,p}) \\
& = \max_{\substack{i,j}} (N_{i,j,\xi,\eta}^{p,p}, \sum_{k,l}
N_{k,l,\xi,\eta}^{p,p}) = \max_{\substack{i,j}} (N_{i,j,\xi,\eta}^{p,p}, 1)
\quad \quad \Big( \text{since }\sum_{k,l} N_{k,l,\xi,\eta}^{p,p}=1 \Big) \\
& = \max_{\substack{i,j}} \int_0^1 \int_0^1 N_{i,j,\xi,\eta}^{p,p} d\xi d\eta.
\end{split}
\end{equation*}
Now,
\begin{equation*}
\begin{split}
& \int_0^1 \int_0^1 N_{i,j,\xi,\eta}^{p,p} d\xi d\eta =\int_0^1 \int_0^1
(-1)^{i+j} {p \choose i}
{p \choose j} \xi^i \eta^j (\xi-1)^{p-i} (\eta-1)^{p-j}d\xi d\eta\\
& = {p \choose i}
{p \choose j} \left(\int_0^1 \xi^{(i+1)-1} (1-\xi)^{(p-i+1)-1}d\xi \right)
\left( \int_0^1 \eta^{(j+1)-1} (1-\eta)^{(p-j+1)-1} d\eta\right)\\
& = {p \choose i}
{p \choose j} \left( \frac{\Gamma(i+1) \Gamma (p-i+1)}{\Gamma (p+2)}\right)
\left(\frac{\Gamma(j+1) \Gamma (p-j+1)}{\Gamma (p+2)}\right)\\
& = \frac{p!}{i!(p-i)!} \frac{p!}{j!(p-j)!} \frac{i!(p-i)!}{(p+1)!}
\frac{j!(p-j)!}{(p+1)!} = \frac{1}{(p+1)^2}.
\end{split}
\end{equation*}
Hence,
\begin{equation*}
\max_{\substack{i,j}} \int_0^1 \int_0^1 N_{i,j,\xi,\eta}^{p,p} d\xi d\eta =
\frac{1}{(p+1)^2}.
\end{equation*}
\hfill\end{proof} 

The symmetry of $M^{e}$ implies
\begin{equation}
\label{eq:l1normM}
\|M^{e}\|_{\infty}= \|M^{e}\|_{1} = \frac{1}{(p+1)^2}.
\end{equation}
\begin{lemma}
\label{le:UBlambdamaxM}
The maximum eigenvalue of the mass matrix $M^{e}$ on the coarsest mesh can be bounded above by
\begin{equation*}
\lambda_{\text{max}}(M^{e}) \leq C \frac{1}{(p+1)^2}.
\end{equation*}
\end{lemma}
\begin{proof}
We have the following inequality for matrix norms
\begin{equation*}
\|M^{e}\|_{2}^2 \leq \|M^{e}\|_{1}\|M^{e}\|_{\infty}.
\end{equation*}
Using Lemma \ref{le:LZidea} and (\ref{eq:l1normM}) we get the bound on the
spectral norm of $M^{e}$,
\begin{equation*}
\|M^{e}\|_{2} \leq C \frac{1}{(p+1)^2}.
\end{equation*}
\hfill\end{proof}

\begin{remark}
In fact, for the coarsest mesh we get $ \lambda_{\text{max}}(M^{e}) =
\displaystyle \frac{1}{(p+1)^2}$ by Lemma \ref{le:MaPoMijBound} and by
\cite[Lemma 2.5]{Varga-96}.
\end{remark}

Using the same argument as in the stiffness matrix case we can give the estimate for the maximum eigenvalue of the global mass matrix using the estimates for the element mass matrices. 

\begin{lemma}
\label{le:GUBlambdamaxM}
The maximum eigenvalue of the global mass matrix $M$ can be bounded above by
\begin{equation*}
\lambda_{\text{max}}(M) \leq C,
\end{equation*}
where $C$ is a constant independent of $p$ (may depend on $h$).
\end{lemma}
\begin{proof}
Following the proof of Lemma \ref{le:Glambdamax}, we have
\begin{equation*}
\begin{split}
\lambda_{\text{max}} (M) & \le C ((p+1)^2) \times \displaystyle \frac{1}{(p+1)^2}.
\end{split}
\end{equation*}
Hence,
\begin{center}
$\lambda_{\text{max}} (M)  \le C.$
\end{center}
\hfill\end{proof}
\begin{remark}
Unlike the stiffness matrix case, this estimate for the mass matrix case is sharp. Since all of the entries of the mass matrix are positive, therefore in the overlapping, the entries of element mass matrices are always added up without any cancellations or reductions.  
\end{remark}

\begin{lemma}
\label{le:LBlambdaminM}
There exists a constant $C$ that is independent of $p$ such that the minimum
eigenvalue of the mass matrix $M$ can be bounded below by
\begin{equation*}
\lambda_{\text{min}}(M) \ge \frac{C}{p^416^p}.
\end{equation*}
\end{lemma}
\begin{proof}
To bound the minimum eigenvalue from below we use the left hand side
inequality of (\ref{eq:BasConP}):
\begin{align*}
\frac{\{v_i\}\cdot M\{v_i\}}{\|\{v_i\}\|^2}= & \frac{(v,v)}{\|\{v_i\}\|^2}\geq
\frac{\displaystyle\frac{C}{p^416^p} \|\{v_i\}\|^2}{\|\{v_i\}\|^2}
= \frac{C}{p^416^p}.
\end{align*}
Therefore,
$\displaystyle \lambda_{\text{min}}(M) \ge \frac{C}{p^416^p},$ where $C$ is a
constant that is independent of $p$.
\hfill\end{proof}

The following lemma gives us the upper bound for the condition number of the
mass matrix.
\begin{lemma}
\label{le:ConBouMasP}
The condition number of the mass matrix $M$ is bounded above by
\begin{equation*}
\kappa(M)\leq Cp^416^p,
\end{equation*}
where $C$ is a constant that is independent of $p$.
\end{lemma}
\begin{proof}
From Lemma \ref{le:UBlambdamaxM} and Lemma \ref{le:LBlambdaminM},
\begin{equation*}
\frac{C}{p^416^p} \leq \lambda_{\text{min}}\leq \lambda_{\text{max}} \leq C.
\end{equation*}
Hence,
\begin{equation*}
\kappa(M)\leq Cp^416^p.
\end{equation*}
\hfill\end{proof}

\begin{remark}
\label{re:dbc_m}
The above bound can be easily generalized for a $d$-dimensional problem:
\begin{equation}
\kappa(M)\le p^{2d}4^{pd}.
\end{equation}
Following Remark \ref{re:dbc} and using de Boor's conjecture
(\ref{eq:CondNumConjecture}), the upper bound for the condition number of the
mass matrix can be further improved and given by
\begin{equation}
\kappa(M) \le \displaystyle 4^{pd}.
\end{equation}
\end{remark}

\begin{remark}
We have done all the analysis on the parametric domain $(0,1)^2$.
To get the results for the physical domain we can define an invertible NURBS
geometrical map from the parametric domain to the physical domain.
With suitable transformations we get the results for the physical domain.
For details, see \cite{BVCHughesS-06}.
\end{remark}


\section{Numerical results}
\label{NumericalResults}

In this section, we provide the numerical results for $h$-refined 
(in Section~\ref{Subsec:Numerical-h})
and
$p$-refined 
(in Section~\ref{Subsec:Numerical-p})
stiffness and mass matrices.
The numerical discretizations are performed using the Matlab toolbox GeoPDEs
\cite{FalcoRV-11,GeoPDEs}.


\subsection{$h$-refinement}
\label{Subsec:Numerical-h}

For $h$-refinement, the condition number of the stiffness matrix is shown
in Table \ref{tab:condHC0p}.
Numerical results are provided from $p=2$ to $p=5$.
In the classical finite element method, the condition number of the stiffness
matrix is of order $h^{-2}$ even for a coarse mesh-size.
However, in isogeometric discretizations, for higher $p$ on coarse mesh, the
condition number is highly influenced by the stability constant of B-Splines.
The condition number of B-Splines heavily depends on the polynomial degree
(see Section~\ref{Sec:SplineCondNum}) and scales as $(p2^p)^d$.
The factor $(p2^p)^d$ dominates the factor $h^{-2}$ for coarse meshes.
Nevertheless, the numerical results support the theoretical findings
asymptotically (for reasonably refined meshes) for any polynomial degree.

\begin{table}[!b]
\caption{ Condition number of the stiffness matrix $A$}
\label{tab:condHC0p}
\begin{center}
\begin{tabular}{|c|c|c|c|c|c|c|c|}\hline
\backslashbox{$p$}{$h^{-1}$} & 2 & 4 & 8 & 16 & 32 & 64 & 128 \\ \hline
2 & 4.00  & 4.00 & 5.22 & 19.77 & 78.14 & 311.58 & 1245.36 \\\hline
3 & 30.93 &  29.51 & 29.19& 28.56 & 82.10 & 327.21 & 1307.67 \\\hline
4 & 339.92 &  269.23 & 240.03 & 222.55 & 215.00 & 381.73 & 1525.40 \\\hline
5  & 4177.20 & 3220.60 &  2148.25 & 1812.58 &  1700.63 & 1688.11 & 1781.51\\\hline
\end{tabular}
\end{center}
\end{table}

In Table~\ref{tab:condHC0pM}, we present the condition number of the mass
matrix.
We see that the condition number is bounded uniformly by a constant
independent of $h$, which confirms the theoretical estimates.

\begin{table}[!b]
\caption{ Condition number of the mass matrix $M$}
\label{tab:condHC0pM}
\begin{center}
\begin{tabular}{|c|c|c|c|c|c|c|c|}\hline
\backslashbox{$p$}{$h^{-1}$}  & 2 & 4 & 8 & 16 & 32 & 64 &128 \\\hline
2 &  89.679& 109.68& 108.51& 109.85& 111.29 & 111.69 & 111.79\\\hline
3 &  915.558& 799.941& 737.379& 708.010& 715.89 & 719.45 & 720.33 \\\hline
4 &  11773.17& 6795.46& 5381.96& 4762.53& 4750.07 & 4779.41 & 4786.90 \\\hline
5 &  163371.70& 77448.11& 42580.04& 33560.40& 32587.27 & 32808.69 & 32871.70 \\\hline
\end{tabular}
\end{center}
\end{table}


\subsection{$p$-refinement}
\label{Subsec:Numerical-p}

We perform numerical experiments for $p$-refinement
to obtain the maximum and minimum
eigenvalues, and the condition number of the stiffness matrix and the mass
matrix.
The eigenvalues and the condition number are obtained on the coarsest mesh and
the finest mesh.
For higher $p$ ($p>10$) roundoff errors start contaminating the results and we
stop reporting with $10$. 

In Tables~\ref{tab:condP1} and \ref{tab:condP1F}, we present the extremal
eigenvalues and the condition number of the stiffness matrix for $p=2$ to
$p=10$.
We observe that the maximum eigenvalue scales as a constant independent of $p$ for the coarsest mesh
and linearly dependent on $p$ for refined meshes, and that the minimum eigenvalue is bounded from below by the bound given in
Theorem \ref{th:ConBouP}.

The extremal eigenvalues and the condition number of the mass matrix for $p=2$ to $p=10$ are presented in Table \ref{tab:condP1M} and Table
\ref{tab:condP1MF}.
Numerical results confirm the theoretical estimates given in Lemma
\ref{le:UBlambdamaxM}, Lemma \ref{le:GUBlambdamaxM}, Lemma \ref{le:LBlambdaminM}, and Lemma
\ref{le:ConBouMasP}.

\begin{table}[!h]
\caption{ $\lambda_{\text{max}}$, $\lambda_{\text{min}}$, and $\kappa(A)$  on the coarsest mesh}
\label{tab:condP1}
\begin{center}
\begin{tabular}{|c|c|c|c|}
\hline
$p$ & $\lambda_{\text{max}}$ & $\lambda_{\text{min}}$ & $\kappa(A)$\\\hline
2   & 0.35 & 3.5e-01 & 1.0e+00 \\\hline
3   & 0.45 & 3.8e-02 & 1.1e+01\\\hline
4   & 0.41 & 2.9e-03 & 1.3e+02\\\hline
5   & 0.35 & 2.1e-04 & 1.6e+03\\\hline
6   & 0.33 & 1.5e-05 & 2.1e+04\\\hline
7   & 0.33 & 1.1e-06 & 2.9e+05\\\hline
8   & 0.31 & 7.8e-08 & 4.0e+06\\\hline
9   & 0.30 & 5.4e-09 & 5.6e+07\\\hline
10 & 0.30 & 3.7e-10 & 8.1e+08\\\hline
\end{tabular}
\end{center}
\end{table}

\begin{table}
\caption{ $\lambda_{\text{max}}$, $\lambda_{\text{min}}$, and $\kappa(A)$  on the finest mesh}
\label{tab:condP1F}
\begin{center}
\begin{tabular}{|c|c|c|c|}
\hline
$p$ & $\lambda_{\text{max}}$ & $\lambda_{\text{min}}$ & $\kappa(A)$\\\hline
2   & 1.50  & 1.2e-03& 1.2e+03 \\\hline
3    & 1.58 & 1.2e-03& 1.3e+03\\\hline
4    & 1.84& 1.2-03 & 1.5e+03\\\hline
5     & 2.14& 1.2e-03  & 1.8e+03\\\hline
6   & 2.47  & 1.8e-04 & 1.4e+04\\\hline
7   & 2.80 & 2.6e-05  & 1.1e+05\\\hline
8     & 3.14& 3.6e-06  & 8.8e+05\\\hline
9    & 3.47& 4.9e-07  & 7.1e+06\\\hline
10   & 3.81& 6.6e-08 & 5.7e+07\\\hline
\end{tabular}
\end{center}
\end{table}

\begin{table}[!b]
\caption{ $\lambda_{\text{max}}$, $\lambda_{\text{min}}$, and $\kappa(M)$  on the coarsest mesh}
\label{tab:condP1M}
\begin{center}
\begin{tabular}{|c|c|c|c|}
\hline
$p$ & $\lambda_{\text{max}}$ & $\lambda_{\text{min}}$ & $\kappa(M)$\\\hline
2   & 1.1e-01& 1.1e-03 &1.0e+02\\\hline
3   & 6.2e-02& 5.1e-05 &1.2e+03 \\\hline
4   & 4.0e-02& 2.5e-06 & 1.5e+04\\\hline
5   &2.7e-02& 1.3e-07& 2.1e+05 \\\hline
6   & 2.0e-02& 6.9e-09 &2.9e+06\\\hline
7   & 1.5e-02& 3.7e-10 & 4.1e+07\\\hline
8   & 1.2e-02& 2.0e-11 & 5.9e+08\\\hline
9   & 1.0e-02& 1.1e-12& 8.5e+09\\\hline
10 & 8.2e-03&6.6e-14& 1.2e+11 \\\hline
\end{tabular}
\end{center}
\end{table}

\begin{table}[!t]
\caption{ $\lambda_{\text{max}}$, $\lambda_{\text{min}}$, and $\kappa(M)$  on the finest mesh}
\label{tab:condP1MF}
\begin{center}
\begin{tabular}{|c|c|c|c|}
\hline
$p$ & $\lambda_{\text{max}}$ & $\lambda_{\text{min}}$ & $\kappa(M)$\\\hline
2   & 6.1e-05& 5.5e-07  & 1.1e+02  \\\hline
3  & 6.1e-05& 8.5e-08  & 7.2e+02  \\\hline
4  & 6.1e-05 & 1.3e-08 & 4.8e+03 \\\hline
5 & 6.1e-05 & 1.9e-09  & 3.3e+04 \\\hline
6  & 6.1e-05 & 2.6e-10 & 2.3e+05 \\\hline
7  & 6.1e-05& 3.7e-11  & 1.7e+06 \\\hline
8  & 6.1e-05 & 5.1e-12 & 1.2e+07 \\\hline
9  & 6.1e-05 & 6.9e-13 & 8.9e+07 \\\hline
10 & 6.1e-05 & 9.2e-14  & 6.6e+08\\\hline
\end{tabular}
\end{center}
\end{table}


\section{Conclusions}
\label{Sec:Conclusions}

We have provided the bounds for the minimum eigenvalue, maximum eigenvalue,
and the condition numbers of the stiffness and mass matrices for the Laplace
operator with $h$- and $p$-refinements of the isogeometric discretizations
that are based on B-Spline (NURBS) basis functions.
We proved that in the $h$-refinement case, like the classical finite element
method, the condition number of the stiffness matrix scales as $h^{-2}$.
For the mass matrix, it scales as constant independent of $h$.
For the $p$-refinement case, we proved that the condition number of the stiffness
and mass matrices grow exponentially in $p$.

The estimates for the minimum eigenvalues of the stiffness and mass
matrices depend on the stability constant of B-Splines.
In reaching these estimates we have used the stability constant of B-Splines
as $p2^p$.
  Using
the de Boor's conjecture (the stability constant of B-Splines given by $2^p$,
which is the best known bound), these estimates can be
further improved according to Remarks~\ref{re:dbc} and \ref{re:dbc_m}.

Unfortunately, a sharp estimate for the stability constant is unknown.
Therefore, a sharp estimate for the minimum eigenvalue cannot be determined
at this time and will be the subject of future research by us and others.
It is a very difficult problem.


\section*{Acknowledgments}

The authors would like to thank Prof. U. Langer (Johannes Kepler University
Linz, Austria) and 
Prof. L. Zikatanov (Pennsylvania Stae University, USA) for helpful suggestions
on the topic of this paper.

\end{document}